\documentclass{article}%
\usepackage{amsfonts}
\usepackage{graphicx}
\usepackage{amsmath}%
\usepackage{amssymb}
\providecommand{\U}[1]{\protect\rule{.1in}{.1in}}
\newtheorem{theorem}{Theorem}

\newtheorem{corollary}[theorem]{Corollary}

\newtheorem{definition}[theorem]{Definition}

\newtheorem{notation}[theorem]{Notation}

\newtheorem{proposition}[theorem]{Proposition}

\newenvironment{proof}[1][Proof]{\textbf{#1.} }{\ \rule{0.5em}{0.5em}}
\begin{document}

\title{Axiomatic Differential Geometry III-2\\-Its Landscape-\\Chapter 2: Model Theory II}
\author{Hirokazu NISHIMURA\\Institute of Mathematics\\University of Tsukuba\\Tsukuba, Ibaraki, 305-8571, JAPAN}
\maketitle

\begin{abstract}
Given a complete and (locally) cartesian closed category $\mathbf{U}$, it is
shown that the category of functors from the category of Weil algebras to the
category $\mathbf{U}$\ is (locally, resp.) cartesian closed. The corresponding
axiomatization for differential geometry based upon Weil functors is then given.

\end{abstract}

\section{\label{s0}Introduction}

\textit{Cartesian closedness} is one of the desirable properties that every
good category is expected to possess. Indeed, it is surely behind Steenrod's
epoch-making notion of a \textit{convenient category} of topological spaces,
for which the reader is referred to \cite{st}. Unlike many other desirable
properties (e.g., completeness), cartesian closedness is not stable under
slicing, and \textit{slicing} within the realm of category theory corresponds
to \textit{fibered manifolds} within the realm of differential geometry.
Therefore the importance of \textit{locally cartesian closedness} in the arena
of differential geometry could not be exaggerated. A few convenient categories
of smooth spaces have been proposed (cf. \cite{baez} and \cite{stacey} for
their panoramic expositions), but not all of them are locally cartesian
closed. By way of example, the category of Chen spaces (cf. \cite{chen3})\ and
that of Souriau spaces (cf. \cite{sou})\ are locally cartesian closed, while
that of Fr\"{o}licher spaces (cf. \cite{fro1} and \cite{fro2})\ is not.

The principal objective in this paper. as a sequel to \cite{nishi9}, is to
show that, given a category $\mathbf{U}$\ which is complete and (locally,
resp.) cartesian closed, the category $\mathcal{K}_{\mathbf{U}}$\ of functors
on the category of Weil algebras to $\mathbf{U}$\ is not only complete but
also (locally, resp.) cartesian closed, which will be explained in \S \ref{s3}
and \S \ref{s4}. A corresponding axiomatization is given in \S \ref{s5}.

\section{Preliminaries}

\subsection{Category Theory}

Given a category $\mathbf{C}$\ and a morphism
\[
f:A\rightarrow B
\]
in $\mathbf{C}$, we write
\begin{align*}
A  & =\mathrm{dom}\,f\\
B  & =\mathrm{cod}\,f
\end{align*}

\subsection{Weil Algebras}

Let $k$\ be a commutative ring. The category of Weil algebras over $k$\ (also
called Weil $k$-algebras) is denoted by $\mathbf{Weil}_{k}$. It is well known
that the category $\mathbf{Weil}_{k}$\ is left exact. The initial and terminal
object in $\mathbf{Weil}_{k}$\ is $k$ itself. Given two objects $W_{1}$ and
$W_{2}$ in the category $\mathbf{Weil}_{k}$, we denote their tensor algebra by
$W_{1}\otimes_{k}W_{2}$. For a good treatise on Weil algebras, the reader is
referred to \S \ 1.16 of \cite{kock}. Given a left exact category
$\mathcal{K}$\ and a $k$-algebra object $\mathbb{R}$\ in $\mathcal{K}$, there
is a canonical functor $\mathbb{R}\underline{\otimes}_{k}\cdot$\ (denoted by
$\mathbb{R\otimes}\mathbb{\cdot}$\ in \cite{kock}) from the category
$\mathbf{Weil}_{k}$ to the category of $k$-algebra objects and their
homomorphisms in $\mathcal{K}$.

\section{\label{s2}The Main Example}

Let $\mathbf{U}$\ be a complete and cartesian closed category with
$\mathbb{R}$\ being a $k$-algebra object in $\mathbf{U}$. We have in mind a
convenient category of smooth spaces as $\mathbf{U}$.

\begin{notation}
We introduce the following notation:

\begin{enumerate}
\item We denote by $\mathcal{K}_{\mathbf{U}}$\ the category whose objects are
functors from the category $\mathbf{Weil}_{k}$ to the category $\mathbf{U}$
and whose morphisms are their natural transformations.

\item Given an object $W$ in the category $\mathbf{Weil}_{k}$, we denote by
\[
\mathbf{T}_{\mathbf{U}}^{W}:\mathcal{K}_{\mathbf{U}}\rightarrow\mathcal{K}%
_{\mathbf{U}}%
\]
the functor obtained as the composition with the functor
\[
\_\otimes_{k}W:\mathbf{Weil}_{k}\rightarrow\mathbf{Weil}_{k}%
\]
so that for any object $M$\ in the category $\mathcal{K}_{\mathbf{U}}$, we
have
\[
\mathbf{T}_{\mathbf{U}}^{W}\left(  M\right)  =M\left(  \_\otimes_{k}W\right)
\]

\item Given a morphism $\varphi:W_{1}\rightarrow W_{2}$ in the category
$\mathbf{Weil}_{k}$, we denote by
\[
\alpha_{\varphi}^{\mathbf{U}}:\mathbf{T}_{\mathbf{U}}^{W_{1}}\Rightarrow
\mathbf{T}_{\mathbf{U}}^{W_{2}}%
\]
the natural transformation such that, given an object $W$\ in the category
$\mathbf{Weil}_{k}$, the morphism
\[
\alpha_{\varphi}^{\mathbf{U}}\left(  M\right)  :\mathbf{T}_{\mathbf{U}}%
^{W_{1}}\left(  M\right)  \rightarrow\mathbf{T}_{\mathbf{U}}^{W_{2}}\left(
M\right)
\]
is
\[
M\left(  W\otimes_{k}\varphi\right)  :M\left(  W\otimes_{k}W_{1}\right)
\rightarrow M\left(  W\otimes_{k}W_{2}\right)
\]

\item We denote by $\mathbb{R}_{\mathbf{U}}$ the functor
\[
\mathbb{R}\underline{\mathbb{\otimes}}_{k}\_:\mathbf{Weil}_{k}\rightarrow
\mathbf{U}%
\]

\end{enumerate}
\end{notation}

\section{\label{s3}Cartesian Closedness}

\begin{theorem}
\label{t3.1}The category $\mathcal{K}_{\mathbf{U}}$\ is cartesian closed.
\end{theorem}

\begin{proof}
The proof is a modification of Exercise 1.3.7\ in \cite{jac}. Let $M$\ and $N
$\ be objects in the category $\mathcal{K}_{\mathbf{U}}$. Given an object $W$
in the category $\mathbf{Weil}_{k}$, we let $M^{N}\left(  W\right)  $\ denote
the intersection of all the equalizers
\begin{align*}
& \prod_{\mathrm{dom}\,\varphi=W}M\left(  \mathrm{cod}\,\varphi\right)
^{N\left(  \mathrm{cod}\,\varphi\right)  }
\begin{array}
[c]{c}%
\rightarrow M\left(  \mathrm{cod}\,\varphi_{1}\right)  ^{N\left(
\mathrm{cod}\,\varphi_{1}\right)  }\,\underrightarrow{M\left(  \varphi
_{2}\right)  ^{N\left(  \mathrm{cod}\,\varphi_{1}\right)  }}\\
\rightarrow M\left(  \mathrm{cod}\,\varphi_{2}\circ\varphi_{1}\right)
^{N\left(  \mathrm{cod}\,\varphi_{2}\circ\varphi_{1}\right)  }%
\,\overrightarrow{M\left(  \mathrm{cod}\,\varphi_{2}\circ\varphi_{1}\right)
^{N\left(  \varphi_{2}\right)  }}%
\end{array}
\\
& M\left(  \mathrm{cod}\,\varphi_{2}\circ\varphi_{1}\right)  ^{N\left(
\mathrm{cod}\,\varphi_{1}\right)  }%
\end{align*}
where $\varphi$\ ranges over all morphisms in the category $\mathbf{Weil}_{k}%
$\ with $\mathrm{dom}\,\varphi=W$, $\varphi_{1}$\ and $\varphi_{2}$ range over
all morphisms in the category $\mathbf{Weil}_{k}$\ with $\mathrm{dom}%
\,\varphi_{1}=W$ and $\mathrm{cod}\,\varphi_{1}=\mathrm{dom}\,\varphi_{2}$,
and the morphisms
\[
\prod_{\mathrm{dom}\,\varphi=W}M\left(  \mathrm{cod}\,\varphi\right)
^{N\left(  \mathrm{cod}\,\varphi\right)  }\rightarrow M\left(  \mathrm{cod}%
\,\varphi_{1}\right)  ^{N\left(  \mathrm{cod}\,\varphi_{1}\right)  }%
\]
and
\[
\prod_{\mathrm{dom}\,\varphi=W}M\left(  \mathrm{cod}\,\varphi\right)
^{N\left(  \mathrm{cod}\,\varphi\right)  }\rightarrow M\left(  \mathrm{cod}%
\,\varphi_{2}\circ\varphi_{1}\right)  ^{N\left(  \mathrm{cod}\,\varphi
_{2}\circ\varphi_{1}\right)  }%
\]
are the canonical projections. Given a morphism $\psi:W_{1}\rightarrow W_{2} $
in the category $\mathbf{Weil}_{k}$, the canonical morphism
\begin{align*}
\prod_{\mathrm{dom}\,\varphi=W_{1}}M\left(  \mathrm{cod}\,\varphi\right)
^{N\left(  \mathrm{cod}\,\varphi\right)  }  & \rightarrow\\
\prod_{\mathrm{dom}\,\varphi^{\prime}=W_{2}}M\left(  \mathrm{cod}%
\,\varphi^{\prime}\circ\psi\right)  ^{N\left(  \mathrm{cod}\,\varphi^{\prime
}\circ\psi\right)  }  & =\prod_{\mathrm{dom}\,\varphi^{\prime}=W_{2}}M\left(
\mathrm{cod}\,\varphi^{\prime}\right)  ^{N\left(  \mathrm{cod}\,\varphi
^{\prime}\right)  }%
\end{align*}
naturally gives rise to a morphism $M^{N}\left(  W_{1}\right)  \rightarrow
M^{N}\left(  W_{2}\right)  $ in the category $\mathcal{K}_{\mathbf{U}}$, which
we let $M^{N}\left(  \psi\right)  $. It is easy to see that $M^{N}$\ becomes
an object in the category $\mathcal{K}_{\mathbf{U}}$, which works as the
exponentiation of $M$\ by $N$ within the category $\mathcal{K}_{\mathbf{U}}$.
\end{proof}

\begin{corollary}
\label{t3.1.1}Given an object $W$ in the category $\mathbf{Weil}_{k}$ and
objects $M$\ and $N$\ in the category $\mathcal{K}_{\mathbf{U}}$, we have
\[
\mathbf{T}_{\mathbf{U}}^{W}\left(  M^{N}\right)  =\mathbf{T}_{\mathbf{U}}%
^{W}\left(  M\right)  ^{\mathbf{T}_{\mathbf{U}}^{W}\left(  N\right)  }%
\]

\end{corollary}

\begin{corollary}
\label{t3.1.2}Given a morphism $\varphi:W_{1}\rightarrow W_{2}$ in the
category $\mathbf{Weil}_{k}$ and objects $M$\ and $N$\ in the category
$\mathcal{K}_{\mathbf{U}}$, the morphism
\[
\alpha_{\varphi}^{\mathbf{U}}\left(  M\right)  ^{\mathbf{T}_{\mathbf{U}%
}^{W_{1}}\left(  N\right)  }:\mathbf{T}_{\mathbf{U}}^{W_{1}}\left(
M^{N}\right)  =\mathbf{T}_{\mathbf{U}}^{W_{1}}\left(  M\right)  ^{\mathbf{T}%
_{\mathbf{U}}^{W_{1}}\left(  N\right)  }\rightarrow\mathbf{T}_{\mathbf{U}%
}^{W_{2}}\left(  M\right)  ^{\mathbf{T}_{\mathbf{U}}^{W_{1}}\left(  N\right)
}%
\]
is equal to the morphism
\begin{align*}
\mathbf{T}_{\mathbf{U}}^{W_{1}}\left(  M\right)  ^{\mathbf{T}_{\mathbf{U}%
}^{W_{1}}\left(  N\right)  }  & =\mathbf{T}_{\mathbf{U}}^{W_{1}}\left(
M^{N}\right)  \,\underrightarrow{\alpha_{\varphi}^{\mathbf{U}}\left(
M^{N}\right)  }\,\mathbf{T}_{\mathbf{U}}^{W_{2}}\left(  M^{N}\right)
=\mathbf{T}_{\mathbf{U}}^{W_{2}}\left(  M\right)  ^{\mathbf{T}_{\mathbf{U}%
}^{W_{2}}\left(  N\right)  }\\
& \underrightarrow{\mathbf{T}_{\mathbf{U}}^{W_{2}}\left(  M\right)
^{\alpha_{\varphi}^{\mathbf{U}}\left(  N\right)  }}\,\mathbf{T}_{\mathbf{U}%
}^{W_{2}}\left(  M\right)  ^{\mathbf{T}_{\mathbf{U}}^{W_{1}}\left(  N\right)
}%
\end{align*}

\end{corollary}

\section{\label{s4}Locally Cartesian Closedness}

In this section we assume that the category $\mathbf{U}$\ is locally cartesian closed.

\begin{theorem}
\label{t4.1}The category $\mathcal{K}_{\mathbf{U}}$\ is locally cartesian closed.
\end{theorem}

\begin{proof}
Our present discussion is a localization of the discussion in the proof of
Theorem \ref{t3.1} in a sense. Let $L$\ be an object in the category
$\mathcal{K}_{\mathbf{U}}$. Let $\pi_{1}:M\rightarrow L$\ and $\pi
_{2}:N\rightarrow L$\ be objects in the slice category $\mathcal{K}%
_{\mathbf{U}}/L$. Given an object $W$ in the category $\mathbf{Weil}_{k}$, we
let $\left(  M^{N}\right)  _{L}\left(  W\right)  $\ denote the intersection of
all the equalizers
\begin{align*}
& \left(  \prod_{\mathrm{dom}\,\varphi=W}M\left(  \mathrm{cod}\,\varphi
\right)  ^{N\left(  \mathrm{cod}\,\varphi\right)  }\right)  _{L}\\
&
\begin{array}
[c]{c}%
\rightarrow\left(  M\left(  \mathrm{cod}\,\varphi_{1}\right)  ^{N\left(
\mathrm{cod}\,\varphi_{1}\right)  }\right)  _{L}\,\underrightarrow{\left(
M\left(  \varphi_{2}\right)  ^{N\left(  \mathrm{cod}\,\varphi_{1}\right)
}\right)  _{L}}\\
\rightarrow\left(  M\left(  \mathrm{cod}\,\varphi_{2}\circ\varphi_{1}\right)
^{N\left(  \mathrm{cod}\,\varphi_{2}\circ\varphi_{1}\right)  }\right)
_{L}\,\overrightarrow{\left(  M\left(  \mathrm{cod}\,\varphi_{2}\circ
\varphi_{1}\right)  ^{N\left(  \varphi_{2}\right)  }\right)  _{L}}%
\end{array}
\\
& \left(  M\left(  \mathrm{cod}\,\varphi_{2}\circ\varphi_{1}\right)
^{N\left(  \mathrm{cod}\,\varphi_{1}\right)  }\right)  _{L}%
\end{align*}
where $\varphi$\ ranges over all morphisms in the category $\mathbf{Weil}_{k}%
$\ with $\mathrm{dom}\,\varphi=W$, $\varphi_{1}$\ and $\varphi_{2}$ range over
all morphisms in the category $\mathbf{Weil}_{k}$\ with $\mathrm{dom}%
\,\varphi_{1}=W$ and $\mathrm{cod}\,\varphi_{1}=\mathrm{dom}\,\varphi_{2}$,
and the morphisms
\[
\left(  \prod_{\mathrm{dom}\,\varphi=W}M\left(  \mathrm{cod}\,\varphi\right)
^{N\left(  \mathrm{cod}\,\varphi\right)  }\right)  _{L}\rightarrow\left(
M\left(  \mathrm{cod}\,\varphi_{1}\right)  ^{N\left(  \mathrm{cod}%
\,\varphi_{1}\right)  }\right)  _{L}%
\]
and
\[
\left(  \prod_{\mathrm{dom}\,\varphi=W}M\left(  \mathrm{cod}\,\varphi\right)
^{N\left(  \mathrm{cod}\,\varphi\right)  }\right)  _{L}\rightarrow\left(
M\left(  \mathrm{cod}\,\varphi_{2}\circ\varphi_{1}\right)  ^{N\left(
\mathrm{cod}\,\varphi_{2}\circ\varphi_{1}\right)  }\right)  _{L}%
\]
are the canonical projections, and $\left(  \mathcal{-}\right)  _{L}$ denotes
the categorical operation within the slice category $\mathcal{K}_{\mathbf{U}%
}/L$\ so that $\left(  M\times N\right)  _{L}$ denotes the fibered product
$M\times_{L}N$ by way of example. Given a morphism $\psi:W_{1}\rightarrow
W_{2}$ in the category $\mathbf{Weil}_{k}$, the canonical morphism
\begin{align*}
\left(  \prod_{\mathrm{dom}\,\varphi=W_{1}}M\left(  \mathrm{cod}%
\,\varphi\right)  ^{N\left(  \mathrm{cod}\,\varphi\right)  }\right)  _{L}  &
\rightarrow\\
\left(  \prod_{\mathrm{dom}\,\varphi^{\prime}=W_{2}}M\left(  \mathrm{cod}%
\,\varphi^{\prime}\circ\psi\right)  ^{N\left(  \mathrm{cod}\,\varphi^{\prime
}\circ\psi\right)  }\right)  _{L}  & =\left(  \prod_{\mathrm{dom}%
\,\varphi^{\prime}=W_{2}}M\left(  \mathrm{cod}\,\varphi^{\prime}\right)
^{N\left(  \mathrm{cod}\,\varphi^{\prime}\right)  }\right)  _{L}%
\end{align*}
naturally gives rise to a morphism $\left(  M^{N}\right)  _{L}\left(
W_{1}\right)  \rightarrow\left(  M^{N}\right)  _{L}\left(  W_{2}\right)  $ in
the slice category $\mathcal{K}_{\mathbf{U}}/L$, which we let $\left(
M^{N}\right)  _{L}\left(  \psi\right)  $. It is easy to see that $\left(
M^{N}\right)  _{L}$\ becomes an object in the slice category $\mathcal{K}%
_{\mathbf{U}}/L$, which works as the exponentiation of $M$\ by $N$ within the
slice category $\mathcal{K}_{\mathbf{U}}/L$.
\end{proof}

\begin{corollary}
\label{t4.1.1}Given an object $W$ in the category $\mathbf{Weil}_{k}$, an
object $L$\ in the category $\mathcal{K}_{\mathbf{U}}$, and objects $\pi
_{1}:M\rightarrow L$\ and $\pi_{2}:N\rightarrow L$\ in the slice category
$\mathcal{K}_{\mathbf{U}}/L$, we have
\[
\left(  \mathbf{T}_{\mathbf{U}}\right)  _{L}^{W}\left(  \left(  M^{N}\right)
_{L}\right)  =\left(  \left(  \mathbf{T}_{\mathbf{U}}\right)  _{L}^{W}\left(
M\right)  ^{\left(  \mathbf{T}_{\mathbf{U}}\right)  _{L}^{W}\left(  N\right)
}\right)  _{L}%
\]
where $\left(  \mathbf{T}_{\mathbf{U}}\right)  _{L}^{W}\left(  M\right)  $
denotes the equalizer of
\[
\mathbf{T}_{\mathbf{U}}^{W}\left(  M\right)
\begin{array}
[c]{c}%
\underrightarrow{\qquad\qquad\qquad\qquad\mathbf{T}_{\mathbf{U}}^{W}\left(
\pi_{1}\right)  \qquad\qquad\qquad\qquad}\\
\overrightarrow{\mathbf{T}_{\mathbf{U}}^{W}\left(  \pi_{1}\right)
}\,\mathbf{T}_{\mathbf{U}}^{W}\left(  L\right)  \,\overrightarrow
{\alpha_{W\rightarrow k}^{\mathbf{U}}\left(  L\right)  }\,\mathbf{T}%
_{\mathbf{U}}^{k}\left(  L\right)  \,\overrightarrow{\alpha_{k\rightarrow
W}^{\mathbf{U}}\left(  L\right)  }%
\end{array}
\mathbf{T}_{\mathbf{U}}^{W}\left(  L\right)
\]
with $W\rightarrow k$\ and $k\rightarrow W$\ being the canonical morphisms in
the category $\mathbf{Weil}_{k}$. We can naturally extend $\left(
\mathbf{T}_{\mathbf{U}}\right)  _{L}^{W}$\ to a functor
\[
\mathcal{K}_{\mathbf{U}}/L\rightarrow\mathcal{K}_{\mathbf{U}}/L
\]
in the sense that, given any commutative diagram
\[%
\begin{array}
[c]{ccc}%
M & \underrightarrow{\,\,f\,\,} & N\\
\pi_{1}\searrow &  & \swarrow\pi_{2}\\
& L &
\end{array}
\]
within the category $\mathcal{K}$, there exists a unique morphism
\[
\left(  \mathbf{T}_{\mathbf{U}}\right)  _{L}^{W}\left(  f\right)  :\left(
\mathbf{T}_{\mathbf{U}}\right)  _{L}^{W}\left(  M\right)  \rightarrow\left(
\mathbf{T}_{\mathbf{U}}\right)  _{L}^{W}\left(  N\right)
\]
making the diagram
\[%
\begin{array}
[c]{ccc}%
\left(  \mathbf{T}_{\mathbf{U}}\right)  _{L}^{W}\left(  M\right)  &
\underrightarrow{\left(  \mathbf{T}_{\mathbf{U}}\right)  _{L}^{W}\left(
f\right)  } & \left(  \mathbf{T}_{\mathbf{U}}\right)  _{L}^{W}\left(  N\right)
\\
\downarrow &  & \downarrow\\
\mathbf{T}_{\mathbf{U}}^{W}\left(  M\right)  & \overrightarrow{\mathbf{T}%
_{\mathbf{U}}^{W}\left(  f\right)  } & \mathbf{T}_{\mathbf{U}}^{W}\left(
N\right)
\end{array}
\]
commutative, where the two vertical arrows are the canonical injections.
\end{corollary}

\begin{corollary}
\label{t4.1.2}Given a morphism $\varphi:W_{1}\rightarrow W_{2}$ in the
category $\mathbf{Weil}_{k}$, an object $L$\ in the category $\mathcal{K}%
_{\mathbf{U}}$, and objects $\pi_{1}:M\rightarrow L$\ and $\pi_{2}%
:N\rightarrow L$\ in the slice category $\mathcal{K}_{\mathbf{U}}/L$, the
morphism
\begin{align*}
\left(  \alpha^{\mathbf{U}}\right)  _{\varphi}^{L}\left(  M\right)  ^{\left(
\mathbf{T}_{\mathbf{U}}\right)  _{L}^{W_{1}}\left(  N\right)  }  & :\left(
\mathbf{T}_{\mathbf{U}}\right)  _{L}^{W_{1}}\left(  \left(  M^{N}\right)
_{L}\right)  =\left(  \left(  \mathbf{T}_{\mathbf{U}}\right)  _{L}^{W_{1}%
}\left(  M\right)  ^{\left(  \mathbf{T}_{\mathbf{U}}\right)  _{L}^{W_{1}%
}\left(  N\right)  }\right)  _{L}\\
& \rightarrow\left(  \left(  \mathbf{T}_{\mathbf{U}}\right)  _{L}^{W_{2}%
}\left(  M\right)  ^{\left(  \mathbf{T}_{\mathbf{U}}\right)  _{L}^{W_{1}%
}\left(  N\right)  }\right)  _{L}%
\end{align*}
is equal to the morphism
\begin{align*}
\left(  \left(  \mathbf{T}_{\mathbf{U}}\right)  _{L}^{W_{1}}\left(  M\right)
^{\left(  \mathbf{T}_{\mathbf{U}}\right)  _{L}^{W_{1}}\left(  N\right)
}\right)  _{L}  & =\left(  \mathbf{T}_{\mathbf{U}}\right)  _{L}^{W_{1}}\left(
\left(  M^{N}\right)  _{L}\right)  \,\underrightarrow{\left(  \alpha
^{\mathbf{U}}\right)  _{\varphi}^{L}\left(  \left(  M^{N}\right)  _{L}\right)
}\,\\
\left(  \mathbf{T}_{\mathbf{U}}\right)  _{L}^{W_{2}}\left(  \left(
M^{N}\right)  _{L}\right)   & =\left(  \left(  \mathbf{T}_{\mathbf{U}}\right)
_{L}^{W_{2}}\left(  M\right)  ^{\left(  \mathbf{T}_{\mathbf{U}}\right)
_{L}^{W_{2}}\left(  N\right)  }\right)  _{L}\,\underrightarrow{\left(
\mathbf{T}_{\mathbf{U}}\right)  _{L}^{W_{2}}\left(  M\right)  ^{\left(
\alpha^{\mathbf{U}}\right)  _{\varphi}^{L}\left(  N\right)  }}\,\\
& \left(  \left(  \mathbf{T}_{\mathbf{U}}\right)  _{L}^{W_{2}}\left(
M\right)  ^{\left(  \mathbf{T}_{\mathbf{U}}\right)  _{L}^{W_{1}}\left(
N\right)  }\right)  _{L}%
\end{align*}
where the natural transformation
\[
\left(  \alpha^{\mathbf{U}}\right)  _{\varphi}^{L}:\left(  \mathbf{T}%
_{\mathbf{U}}\right)  _{L}^{W_{1}}\Rightarrow\left(  \mathbf{T}_{\mathbf{U}%
}\right)  _{L}^{W_{2}}%
\]
is induced by the natural transformation
\[
\alpha_{\varphi}^{\mathbf{U}}:\mathbf{T}_{\mathbf{U}}^{W_{1}}\Rightarrow
\mathbf{T}_{\mathbf{U}}^{W_{2}}%
\]
in the sense of making the diagram
\[%
\begin{array}
[c]{ccc}%
\left(  \mathbf{T}_{\mathbf{U}}\right)  _{L}^{W_{1}}\left(  M\right)  &
\underrightarrow{\left(  \alpha^{\mathbf{U}}\right)  _{\varphi}^{L}\left(
\pi_{1}\right)  } & \mathbf{T}_{\mathbf{U}}^{W_{2}}\left(  M\right) \\
\downarrow &  & \downarrow\\
\mathbf{T}_{\mathbf{U}}^{W_{1}}\left(  M\right)  & \overrightarrow
{\alpha_{\varphi}^{\mathbf{U}}\left(  M\right)  } & \mathbf{T}_{\mathbf{U}%
}^{W_{2}}\left(  M\right)
\end{array}
\]
commutative.
\end{corollary}

\section{\label{s5}The Axiomatics}

\begin{definition}
A \underline{DG-category} (DG stands for Differential Geometry) is a quadruple
$\left(  \mathcal{K},\mathbf{T},\alpha,\mathbb{R}\right)  $, where

\begin{enumerate}
\item $\mathcal{K}$ is a category which is complete and cartesian closed.

\item Given an object $W$ in the category $\mathbf{Weil}_{k}$, $\mathbf{T}%
^{W}:\mathcal{K}\rightarrow\mathcal{K}$ is a functor subject to the conditions:

\begin{itemize}
\item $\mathbf{T}^{W}$ preserves limits.

\item $\mathbf{T}^{k}:\mathcal{K}\rightarrow\mathcal{K}$ is the identity functor.

\item We have
\[
\mathbf{T}^{W_{2}}\circ\mathbf{T}^{W_{1}}=\mathbf{T}^{W_{1}\otimes_{k}W_{2}}%
\]
for any objects $W_{1}$ and $W_{2}$ in the category $\mathbf{Weil}_{k}$.

\item We have
\[
\mathbf{T}^{W}\left(  M^{N}\right)  =\mathbf{T}^{W}\left(  M\right)
^{\mathbf{T}^{W}\left(  N\right)  }%
\]
for any objects $M$\ and $N$\ in the category $\mathcal{K}$.
\end{itemize}

\item Given a morphism $\varphi:W_{1}\rightarrow W_{2}$ in the category
$\mathbf{Weil}_{k}$, $\alpha_{\varphi}:\mathbf{T}^{W_{1}}\Rightarrow
\mathbf{T}^{W_{2}}$ is a natural transformation subject to the conditions:

\begin{itemize}
\item We have
\[
\alpha_{\mathrm{id}_{W}}=\mathrm{id}_{\mathbf{T}^{W}}%
\]
for any identity morphism $\mathrm{id}_{W}:W\rightarrow W$ in the category
$\mathbf{Weil}_{k}$.

\item We have
\[
\alpha_{\psi}\cdot\alpha_{\varphi}=\alpha_{\psi\circ\varphi}%
\]
for any morphisms $\varphi:W_{1}\rightarrow W_{2}$ and $\psi:W_{2}\rightarrow
W_{3}$ in the category $\mathbf{Weil}_{k}$.

\item Given objects $M$\ and $N$\ in the category $\mathcal{K}$, the morphism
\[
\alpha_{\varphi}\left(  M\right)  ^{\mathbf{T}^{W_{1}}\left(  N\right)
}:\mathbf{T}^{W_{1}}\left(  M^{N}\right)  =\mathbf{T}^{W_{1}}\left(  M\right)
^{\mathbf{T}^{W_{1}}\left(  N\right)  }\rightarrow\mathbf{T}^{W_{2}}\left(
M\right)  ^{\mathbf{T}^{W_{1}}\left(  N\right)  }%
\]
is equal to the morphism
\begin{align*}
\mathbf{T}^{W_{1}}\left(  M\right)  ^{\mathbf{T}^{W_{1}}\left(  N\right)  }  &
=\mathbf{T}^{W_{1}}\left(  M^{N}\right)  \,\underrightarrow{\alpha_{\varphi
}\left(  M^{N}\right)  }\,\mathbf{T}^{W_{2}}\left(  M^{N}\right)
=\mathbf{T}^{W_{2}}\left(  M\right)  ^{\mathbf{T}^{W_{2}}\left(  N\right)  }\\
& \underrightarrow{\mathbf{T}^{W_{2}}\left(  M\right)  ^{\alpha_{\varphi
}\left(  N\right)  }}\,\mathbf{T}^{W_{2}}\left(  M\right)  ^{\mathbf{T}%
^{W_{1}}\left(  N\right)  }%
\end{align*}
within the category $\mathcal{K}$.
\end{itemize}

\item Given an object $W$ in the category $\mathbf{Weil}_{k}$, we have
\[
\mathbf{T}^{W}\left(  \mathbb{R}\right)  =\mathbb{R}\underline{\otimes}_{k}W
\]

\item Given a morphism $\varphi:W_{1}\rightarrow W_{2}$ in the category
$\mathbf{Weil}_{k}$, we have
\[
\alpha_{\varphi}\left(  \mathbb{R}\right)  =\mathbb{R}\underline{\otimes}%
_{k}\varphi
\]

\end{enumerate}
\end{definition}

\begin{notation}
Given an object $W$ in the category $\mathbf{Weil}_{k}$, an object $L$\ in the
category $\mathcal{K}$, and an object $\pi:M\rightarrow L$\ in the slice
category $\mathcal{K}/L$, we denote by $\mathbf{T}_{L}^{W}\left(  M\right)  $
the equalizer of
\[
\mathbf{T}^{W}\left(  M\right)
\begin{array}
[c]{c}%
\underrightarrow{\qquad\qquad\qquad\qquad\mathbf{T}^{W}\left(  \pi\right)
\qquad\qquad\qquad\qquad}\\
\overrightarrow{\mathbf{T}^{W}\left(  \pi\right)  }\,\mathbf{T}^{W}\left(
L\right)  \,\overrightarrow{\alpha_{W\rightarrow k}\left(  L\right)
}\,\mathbf{T}^{k}\left(  L\right)  \,\overrightarrow{\alpha_{k\rightarrow
W}\left(  L\right)  }%
\end{array}
\mathbf{T}^{W}\left(  L\right)
\]
with $W\rightarrow k$\ and $k\rightarrow W$\ being the canonical morphisms
within the category $\mathbf{Weil}_{k}$. We can naturally extend
$\mathbf{T}_{L}^{W}$\ to a functor
\[
\mathcal{K}/L\rightarrow\mathcal{K}/L
\]
in the sense that, given any commutative diagram
\[%
\begin{array}
[c]{ccc}%
M & \underrightarrow{\,\,f\,\,} & N\\
\pi_{1}\searrow &  & \swarrow\pi_{2}\\
& L &
\end{array}
\]
within the category $\mathcal{K}$, there exists a unique morphism
\[
\mathbf{T}_{L}^{W}\left(  f\right)  :\mathbf{T}_{L}^{W}\left(  M\right)
\rightarrow\mathbf{T}_{L}^{W}\left(  N\right)
\]
making the diagram
\[%
\begin{array}
[c]{ccc}%
\mathbf{T}_{L}^{W}\left(  M\right)  & \underrightarrow{\mathbf{T}_{L}%
^{W}\left(  f\right)  } & \mathbf{T}_{L}^{W}\left(  N\right) \\
\downarrow &  & \downarrow\\
\mathbf{T}^{W}\left(  M\right)  & \overrightarrow{\mathbf{T}^{W}\left(
f\right)  } & \mathbf{T}^{W}\left(  N\right)
\end{array}
\]
commutative, where the two vertical arrows are the canonical injections.
\end{notation}

\begin{notation}
Given a morphism $\varphi:W_{1}\rightarrow W_{2}$ in the category
$\mathbf{Weil}_{k}$, an object $L$\ in the category $\mathcal{K}$, and an
object $\pi:M\rightarrow L$\ within the slice category $\mathcal{K}/L$, we
denote by $\alpha_{\varphi}^{L}$ the natural transformation
\[
\mathbf{T}_{L}^{W_{1}}\Rightarrow\mathbf{T}_{L}^{W_{2}}%
\]
making the diagram
\[%
\begin{array}
[c]{ccc}%
\mathbf{T}_{L}^{W_{1}}\left(  M\right)  & \underrightarrow{\alpha_{\varphi
}^{L}\left(  M\right)  } & \mathbf{T}_{L}^{W_{2}}\left(  M\right) \\
\downarrow &  & \downarrow\\
\mathbf{T}^{W_{1}}\left(  M\right)  & \overrightarrow{\alpha_{\varphi}\left(
M\right)  } & \mathbf{T}^{W_{2}}\left(  M\right)
\end{array}
\]
commutative for any object $W$ in the category $\mathbf{Weil}_{k}$, where $%
\begin{array}
[c]{c}%
\mathbf{T}_{L}^{W_{1}}\left(  M\right) \\
\downarrow\\
\mathbf{T}^{W_{1}}\left(  M\right)
\end{array}
$\ and $%
\begin{array}
[c]{c}%
\mathbf{T}_{L}^{W_{2}}\left(  M\right) \\
\downarrow\\
\mathbf{T}^{W_{2}}\left(  M\right)
\end{array}
$\ are the canonical injections.
\end{notation}

\begin{definition}
A \underline{local DG-category} is a DG-category $\left(  \mathcal{K}%
,\mathbf{T},\alpha,\mathbb{R}\right)  $ subject to the conditions:

\begin{enumerate}
\item The category $\mathcal{K}$\ is not only cartesian closed but, what is
even more, locally cartesian closed.

\item Given an object $W$ in the category $\mathbf{Weil}_{k}$, an object
$L$\ in the category $\mathcal{K}$, and objects $\pi_{1}:M\rightarrow L$\ and
$\pi_{2}:N\rightarrow L$\ in the slice category $\mathcal{K}/L$, we have
\[
\mathbf{T}_{L}^{W}\left(  \left(  M^{N}\right)  _{L}\right)  =\left(
\mathbf{T}_{L}^{W}\left(  M\right)  ^{\mathbf{T}_{L}^{W}\left(  N\right)
}\right)  _{L}%
\]
within the category $\mathcal{K}/L$.

\item Given an object $W$ in the category $\mathbf{Weil}_{k}$, an object
$L$\ in the category $\mathcal{K}$, and objects $\pi_{1}:M\rightarrow L$\ and
$\pi_{2}:N\rightarrow L$\ in the slice category $\mathcal{K}/L$, the morphism
\[
\alpha_{\varphi}^{L}\left(  M\right)  ^{\mathbf{T}_{L}^{W_{1}}\left(
N\right)  }:\mathbf{T}_{L}^{W_{1}}\left(  \left(  M^{N}\right)  _{L}\right)
=\left(  \mathbf{T}_{L}^{W_{1}}\left(  M\right)  ^{\mathbf{T}_{L}^{W_{1}%
}\left(  N\right)  }\right)  _{L}\rightarrow\left(  \mathbf{T}_{L}^{W_{2}%
}\left(  M\right)  ^{\mathbf{T}_{L}^{W_{1}}\left(  N\right)  }\right)  _{L}%
\]
is equal to the morphism
\begin{align*}
\left(  \mathbf{T}_{L}^{W_{1}}\left(  M\right)  ^{\mathbf{T}_{L}^{W_{1}%
}\left(  N\right)  }\right)  _{L}  & =\mathbf{T}_{L}^{W_{1}}\left(  \left(
M^{N}\right)  _{L}\right)  \,\underrightarrow{\alpha_{\varphi}^{L}\left(
\left(  M^{N}\right)  _{L}\right)  }\,\\
\mathbf{T}_{L}^{W_{2}}\left(  \left(  M^{N}\right)  _{L}\right)   & =\left(
\mathbf{T}_{L}^{W_{2}}\left(  M\right)  ^{\mathbf{T}_{L}^{W_{2}}\left(
N\right)  }\right)  _{L}\,\underrightarrow{\mathbf{T}_{L}^{W_{2}}\left(
M\right)  ^{\alpha_{\varphi}^{L}\left(  N\right)  }}\,\left(  \mathbf{T}%
_{L}^{W_{2}}\left(  M\right)  ^{\mathbf{T}_{L}^{W_{1}}\left(  N\right)
}\right)  _{L}%
\end{align*}
within the category $\mathcal{K}/L$.
\end{enumerate}
\end{definition}

\begin{proposition}
Given a local DG-category $\left(  \mathcal{K},\mathbf{T},\alpha
,\mathbb{R}\right)  $ and an object $L$\ in the category $\mathcal{K}$, the
quadruple
\[
\left(  \mathcal{K}/L,\mathbf{T}_{L},\alpha^{L},
\begin{array}
[c]{c}%
L\times\mathbb{R}\\
\downarrow\\
L
\end{array}
\right)  \text{,}%
\]
which may be considerd to be the \underline{localization of the DG-category}
$\left(  \mathcal{K},\mathbf{T},\alpha,\mathbb{R}\right)  $ \underline{with
respect to} $L$ in a sense, is a local DG-category, where $%
\begin{array}
[c]{c}%
L\times\mathbb{R}\\
\downarrow\\
L
\end{array}
$\ is the canonical projection.
\end{proposition}

\begin{proof}
Given an object $%
\begin{array}
[c]{c}%
M\\
\downarrow\\
L
\end{array}
$\ in the slice category $\mathcal{K}/L$, we note the following:

\begin{enumerate}
\item We can naturally identify the slice category
\[
\left(  \mathcal{K}/L\right)  /\left(
\begin{array}
[c]{c}%
M\\
\downarrow\\
L
\end{array}
\right)
\]
with the slice category
\[
\mathcal{K}/M
\]
for which the reader is referred, say, to Page 8 of \cite{john}.

\item We have
\[
\left(  \mathbf{T}_{L}\right)  _{M\rightarrow L}=\mathbf{T}_{M}\text{,}%
\]
since the diagram
\[%
\begin{array}
[c]{ccccc}%
\mathbf{T}^{W}\left(  M\right)  & \underrightarrow{\alpha_{W\rightarrow
k}\left(  M\right)  } & \mathbf{T}^{k}\left(  M\right)  & \underrightarrow
{\alpha_{k\rightarrow W}\left(  M\right)  } & \mathbf{T}^{W}\left(  M\right)
\\
\downarrow &  & \downarrow &  & \downarrow\\
\mathbf{T}^{W}\left(  L\right)  & \overrightarrow{\alpha_{W\rightarrow
k}\left(  L\right)  } & \mathbf{T}^{k}\left(  L\right)  & \overrightarrow
{\alpha_{k\rightarrow W}\left(  L\right)  } & \mathbf{T}^{W}\left(  L\right)
\end{array}
\]
is commutative.

\item It is easy to see that
\[
\left(  \alpha^{L}\right)  ^{M\rightarrow L}=\alpha^{M}%
\]

\item We have
\begin{align*}
& M\times_{L}\left(  L\times\mathbb{R}\right) \\
& =M\times\mathbb{R}%
\end{align*}

\end{enumerate}

Therefore the localization
\[
\left(  \left(  \mathcal{K}/L\right)  /\left(
\begin{array}
[c]{c}%
M\\
\downarrow\\
L
\end{array}
\right)  ,\left(  \mathbf{T}_{L}\right)  _{M\rightarrow L},\left(  \alpha
^{L}\right)  ^{M\rightarrow L},
\begin{array}
[c]{ccc}%
M\times_{L}\left(  L\times\mathbb{R}\right)  & \rightarrow & L\times
\mathbb{R}\\
\downarrow &  & \downarrow\\
M & \rightarrow & L
\end{array}
\right)
\]
of the DG-category
\[
\left(  \mathcal{K}/L,\mathbf{T}_{L},\alpha^{L},
\begin{array}
[c]{c}%
L\times\mathbb{R}\\
\downarrow\\
L
\end{array}
\right)
\]
with respect to
\[%
\begin{array}
[c]{c}%
M\\
\downarrow\\
L
\end{array}
\]
is no other than the localization
\[
\left(  \mathcal{K}/M,\mathbf{T}_{M},\alpha^{M},
\begin{array}
[c]{c}%
M\times\mathbb{R}\\
\downarrow\\
M
\end{array}
\right)
\]
of the DG-category
\[
\left(  \mathcal{K},\mathbf{T},\alpha,\mathbb{R}\right)
\]
with respect to $M$,\ so that the desired conclusion follows readily.
\end{proof}

\end{document}